\documentclass[12pt]{amsart}
\usepackage{euscript,epsf,amssymb,xr}

\let\oldmarginpar\marginpar
\renewcommand\marginpar[1]
{\oldmarginpar{\tiny\bf \begin{flushleft} #1 \end{flushleft}}}

\input xy
\xyoption{all}

\newcommand{\la}{\langle}
\newcommand{\ra}{\rangle}

\newtheorem{theorem}{\bf Theorem}[section]
\newtheorem{lemma}[theorem]{\bf Lemma}

\newtheorem{corollary}[theorem]{\bf Corollary}

\newcommand{\CC}{{\Bbb C}}

\newcommand{\FF}{{\Bbb F}}

\newcommand{\NN}{{\Bbb N}}
\newcommand{\PP}{{\Bbb P}}
\newcommand{\QQ}{{\Bbb Q}}
\newcommand{\RR}{{\Bbb R}}

\newcommand{\ZZ}{{\Bbb Z}}

\newcommand{\ggreat}{>\kern-.7ex>}
\newcommand{\ssmall}{<\kern-.7ex<}
\newcommand{\qu}{/\kern-.7ex/}
\newcommand{\exh}{\to\kern-1.8ex\to}

\newcommand{\cC}{{\EuScript{C}}}

\newcommand{\fF}{{\EuScript{F}}}
\newcommand{\gG}{{\EuScript{G}}}
\newcommand{\hH}{{\EuScript{H}}}

\newcommand{\sS}{{\EuScript{S}}}

\newcommand{\GL}{\operatorname{GL}}

\newcommand{\Aut}{\operatorname{Aut}}
\renewcommand{\big}{\operatorname{big}}

\newcommand{\Diff}{\operatorname{Diff}}
\newcommand{\GCD}{\operatorname{GCD}}
\newcommand{\Hom}{\operatorname{Hom}}
\newcommand{\Id}{\operatorname{Id}}

\newcommand{\Jor}{\operatorname{J}}
\newcommand{\Ker}{\operatorname{Ker}}

\renewcommand{\small}{\operatorname{small}}
\newcommand{\SO}{\operatorname{SO}}
\newcommand{\surf}{\operatorname{surf}}
\newcommand{\Tr}{\operatorname{Tr}}

\begin{document}

\title[Finite group actions on $4$-manifolds with $\chi\neq 0$]
{Finite group actions on $4$-manifolds with nonzero Euler characteristic}
\author{Ignasi Mundet i Riera}
\address{Departament d'\`Algebra i Geometria\\
Facultat de Matem\`atiques\\
Universitat de Barcelona\\
Gran Via de les Corts Catalanes 585\\
08007 Barcelona \\
Spain}
\email{ignasi.mundet@ub.edu}

\date{August 27, 2015}

\subjclass[2010]{57S17,54H15}
\thanks{This work has been partially supported by the (Spanish) MEC Project MTM2012-38122-C03-02.}

\maketitle

\begin{abstract}
We prove that if $X$ is a compact, oriented, connected $4$-dimensional
smooth manifold, possibly with boundary,
satisfying $\chi(X)\neq 0$, then there exists a natural number $C$ such that
any finite group $G$ acting smoothly and effectively on $X$ has an abelian
subgroup $A$ generated by two elements which
satisfies $[G:A]\leq C$ and $\chi(X^A)=\chi(X)$. Furthermore, if
$\chi(X)<0$ then $A$ is cyclic.
This answers positively, for any such $X$, a question
of \'Etienne Ghys. We also prove an analogous result for manifolds of arbitrary
dimension and non-vanishing Euler characteristic, but restricted to pseudofree actions.
\end{abstract}

\section{Introduction}

\subsection{}
In this paper we prove two results on smooth finite group actions on compact,
connected manifolds with non-vanishing Euler characteristic, and possibly with boundary.
Our main result is on actions on $4$-dimensional manifolds:

\begin{theorem}
\label{thm:main}
Let $X$ be a compact, orientable, connected $4$-dimensional
smooth manifold, possibly with boundary,
satisfying $\chi(X)\neq 0$. There exists a natural number $C$ such that
any finite group $G$ acting smoothly and effectively on $X$ has an abelian
subgroup $A$ satisfying $[G:A]\leq C$ and $\chi(X^A)=\chi(X)$. Furthermore,
if $\chi(X)>0$ then $A$ can
be generated by $2$ elements, and if $\chi(X)<0$ then $A$ is cyclic.
\end{theorem}

To put Theorem \ref{thm:main} in context, recall the following classic theorem of Camille Jordan
(see \cite{J} and \cite{CR,M1} for modern proofs).
\begin{theorem}[Jordan]
\label{thm:Jordan-classic} For any natural number $n$ there is some constant $\Jor_n$
such that any finite subgroup $G\subset\GL(n,\RR)$ has an abelian subgroup $A$
satisfying $[G:A]\leq \Jor_n.$
\end{theorem}
Let us say that a group $\gG$ is Jordan if there is some
constant $C$ such that any finite subgroup $G\subseteq\gG$ has
an abelian subgroup $A$ satisfying $[G:A]\leq C$ (this
terminology was introduced a few years ago by Popov
\cite{Po0}). It is easy to deduce from Jordan's theorem,
Peter--Weyl's theorem, and the existence and uniqueness up to
conjugation of maximal compact subgroups, that any finite
dimensional Lie group with finitely many connected components
is Jordan.

In the mid 90's, Ghys \cite{G} raised the question of whether
the diffeomorphism group of any compact manifold is Jordan.
This question appeared in print in \cite[Question 13.1]{F}.

The first statement of Theorem \ref{thm:main} gives a positive
answer to Ghys's question for compact connected orientable $4$-manifolds
with nonzero Euler characteristic.
Using the arguments in Subsection 2.3 of \cite{M1}, one can
deduce from Theorem \ref{thm:main} that the diffeomorphism
groups of compact connected nonorientable $4$-manifolds with nonzero Euler
characteristic are Jordan (in both cases connectedness is not a crucial
property, as long as the manifolds are compact and hence have finitely
many connected components).

There are other cases in
which Ghys's question is known to have an affirmative answer.
In \cite{M1} it was proved that if a compact connected $n$-dimensional manifold $X$
admits one-dimensional integral cohomology classes $\alpha_1,\dots,\alpha_n$ whose
product is nonzero then
$\Diff(X)$ is Jordan. This applies for example to tori $T^n$ of
arbitrary dimension.
Zimmermann \cite{Z} proved, using Perelman's proof
of Thurston's geometrization conjecture, that if $X$ is a
compact $3$-manifold then $\Diff(X)$ is Jordan.

In \cite{M2} it was proved that
$\Diff(S^n)$ and $\Diff(\RR^n)$ are Jordan for any $n$; the
paper \cite{M2} also proves that if $X$ is compact and has
nonzero Euler characteristic then $\Diff(X)$ is Jordan. It
should be noted that \cite{M2} uses a result of A. Turull and
the author \cite{MT} which is based on the classification of
finite simple groups (CFSG). In contrast, the present paper
only uses very basic and standard techniques of finite
transformation groups. Note on the other hand that the part of
Theorem \ref{thm:main} which refers the the fixed point set of
the abelian group does not follow from the results in
\cite{M2}.

Roughly one year after the first version of this paper appeared
as a preprint \cite{M3}, Csik\'os, Pyber and Szab\'o \cite{CPS}
proved that $\Diff(T^2\times S^2)$ is not Jordan, thus giving
the first example of a compact manifold whose diffeomorphism
group is not Jordan (previously Popov \cite{Po} had given a
noncompact $4$-dimensional example). It seems to be an
interesting question to understand which compact $4$-manifolds
have Jordan diffeomorphism group (the author does not know any
counterexample which is not an $S^2$-fibration over $T^2$).

Using more sophisticated methods than the present
paper, McCooey has proved in \cite{Mc1,Mc2} very strong
restrictions on finite groups acting effectively and
homologically trivially on general compact, oriented, connected
and closed $4$-manifolds satisfying $\chi\neq 0$. In
particular, the main theorem in \cite{Mc1} implies that if $X$
is a compact simply connected $4$-manifold then $\Diff(X)$ is
Jordan. The paper \cite{Mc2} contains results on actions on non
simply connected compact $4$-manifolds, but these results
require, besides the homological triviality of the action, some
technical restrictions on the manifold, or on the finite group
which acts on it, or on the action, so they do not seem to
apply to all actions of finite groups on closed $4$-manifolds
with nonzero Euler characteristic.

For other results, implying a positive answer to Ghys's question for
$4$-manifolds with vanishing first homology and
$b_2\leq 2$, see e.g. \cite{HL,MZ1,MZ2,W}.

The following is an immediate consequence of Theorem \ref{thm:main}.

\begin{corollary}
Let $X$ be a compact, oriented, connected $4$-dimensional
smooth manifold, possibly with boundary,
satisfying $\chi(X)\neq 0$. There exist constants $C,C'$ with the following properties.
\begin{enumerate}
\item Any finite group acting effectively on $X$ can be generated by $C$ elements.
\item For any action of a finite group $G$ on $X$ there exists some point $x\in X$ whose
isotropy group satisfies $[G:G_x]\leq C'$.
\end{enumerate}
\end{corollary}

Using group theoretical results based on the CFSG,
one can prove the first part of the previous corollary for any compact manifold $X$.
Since to the best of the author's knowledge this has not appeared in the literature, we
briefly explain the argument\footnote{I thank A. Jaikin and E. Khukhro for explaining this
argument to me.}.
By the main result in \cite{MS}, there exists an integer $r$ such that, for any prime $p$,
any elementary $p$-group acting effectively on $X$ has rank at most $r$.
Suppose that $\Gamma$ is a $p$-group acting effectively on $X$; let $\Gamma_0$ be a maximal
abelian normal subgroup of $\Gamma$. The action by conjugation identifies $\Gamma/\Gamma_0$
with a subgroup of $\Aut(\Gamma_0)$. Since $\Gamma_0$ can be generated by at most $r$ elements,
the Gorchakov--Hall--Merzlyakov--Roseblade lemma (see e.g. Lemma 5 in \cite{Ro}) implies
that $\Gamma/\Gamma_0$ can be generated by at most $r(5r-1)/2$ elements.
Hence $\Gamma$ can be generated by at most $r(5r+1)/2$ elements.
According to a theorem proved independently by Guralnick and Lucchini \cite{Gu,L},
if all Sylow subgroups of a finite group $G$ can be generated by at most $k$ elements,
then $G$ itself can be generated by at most $k+1$ elements (both \cite{Gu} and \cite{L}
use the CFSG). Hence any finite group
acting effectively on $X$ can be generated by at most $r(5r+1)/2+1$ elements.

Our second result is analogous to the first one. Whereas the
class of manifolds to which it applies is much wider, it is
limited to pseudofree actions. (Recall that an action of a group
$G$ on a manifold $X$ is pseudofree if for any nontrivial $g\in
G$ the fixed points of $g$ are isolated.)

\begin{theorem}
\label{thm:main-isolated}
Let $X$ be a compact connected manifold, possibly with boundary, with nonzero Euler characteristic. There exists a natural number $C$ such that,
if a finite group $G$ acts pseudofreely, smoothly and effectively on $X$, then
$G$ has an abelian
subgroup $A$ satisfying $[G:A]\leq C$ and $\chi(X^A)=\chi(X)$, and $A$ can
be generated by $[\dim X/2]$ elements.
\end{theorem}

This theorem is certainly  far from answering Ghys's question
for the manifolds to which it applies, since the restriction to
pseudofree actions is very strong. A complete proof that these
manifolds have Jordan diffeomorphism group appears in
\cite{M2}. The reason we include this theorem in this paper is
that the proof of Theorem \ref{thm:main-isolated} serves as a
toy model for the proof of Theorem \ref{thm:main} (note that
the proof in \cite{M2} uses the CFSG, while the arguments we
use to prove Theorem \ref{thm:main-isolated} are completely
elementary; in fact, the proof of Theorem
\ref{thm:main-isolated} is similar to a standard proof of the
Hurwitz bound on the size of automorphism groups of Riemann
surfaces of genus $\geq 2$, see \cite[\S V.1.3]{FK}).

\subsection{Conventions, notation, and contents}
By a natural number we understand a strictly positive integer.
The symbol $\subset$ is reserved for strict inclusion.
All manifolds in this paper will implicitly be assumed to be smooth and possibly with boundary,
and all group actions on manifolds will be smooth.
If $G$ is a group and $S_1,\dots,S_r$ are subsets of $G$, $\la S_1,\dots,S_r\ra$
denotes the subgroup of $G$ generated by the elements of $S_1,\dots,S_r$.
When we say that a group $G$ can be generated by $d$ elements we mean that there
are elements $g_1,\dots,g_d\in G$, {\it not necessarily distinct}, which generate $G$.
If a group $G$ acts on a set $X$ we denote
the stabiliser of $x\in X$ by $G_x$, and for any subset $S\subseteq G$ we denote
$X^S=\{x\in X\mid S\subseteq G_x\}$.
If $g\in G$ we write $X^g$ for $X^{\{g\}}$.

We will systematically use this convention: when we say that some quantity is $A$-bounded we mean
that that quantity is bounded above by a function depending only on $A$; here $A$ can either be
a number, a manifold (then the upper bound depends on the diffeomorphism class of $A$), or a
tuple of objects. This will hopefully make the reading lighter, but it will naturally prevent
us from keeping track of the precise value of the bounds we obtain. In any case, due to the elementary nature of our arguments, the bounds that can be deduced are very likely far from optimal.

We close this introduction with a description of the contents of the paper.
Section \ref{s:preliminaries} contains several unrelated results which
will be used in the subsequent sections. Section \ref{s:pseudofree} contains the
proof of Theorem \ref{thm:main-isolated}. Section \ref{s:main} contains the proof of Theorem \ref{thm:main}. The last two sections contain some auxiliary results which are used in
the proof of Theorem \ref{thm:main}: Section \ref{s:surfaces} gathers some results on finite
group actions on surfaces (in particular, Lemma \ref{lemma:surface-Ghys-chi} is the analogue
of Theorem \ref{thm:main} for surfaces), and Section \ref{s:rigid} contains some results on
finite abelian  groups actions on compact $4$-manifolds and on $C$-rigid actions.

\section{Preliminaries}
\label{s:preliminaries}

\subsection{Linearizing group actions}
\label{ss:linearization} The following result is well known
(see e.g. \cite[Lemma 2.1]{M2}). It implies that the fixed
point set of any (smooth) finite group action on a manifold
with boundary is a neat submanifold in the sense of \cite[\S
1.4]{H}.

\begin{lemma}
\label{lemma:linearization}
Let a finite group $G$ act smoothly on a manifold $X$, and let $x\in X^{G}$.
The tangent space $T_xX$ carries a linear action of $G$, defined as the derivative
at $x$ of the action on $X$, satisfying the following properties.
\begin{enumerate}
\item There exist neighborhoods $U\subset T_xX$ and
    $V\subset X$, of $0$ and $x$ resp., such that:
    \begin{enumerate}
    \item if $x\notin\partial X$ then there is a $G$-equivariant diffeomorphism $\phi\colon U\to V$;
    \item if $x\in\partial X$ then there is $G$-equivariant diffeomorphism $\phi\colon U\cap \{\xi\geq 0\}\to V$, where $\xi$ is a nonzero
        $G$-invariant element of $(T_xX)^*$ such that $\Ker\xi=T_x\partial X$.
    \end{enumerate}
\item If the action of $G$ is effective and $X$ is connected then the action of
$G$ on $T_xX$ is effective, so it induces an inclusion $G\hookrightarrow\GL(T_xX)$.
\end{enumerate}
\end{lemma}

\begin{lemma}
\label{lemma:fixed-point-Jordan}
Let a finite group $G$ act smoothly on a connected manifold $X$, and assume that $X^G\neq\emptyset$.
Then $G$ has an abelian subgroup $A$ of $X$-bounded index.
\end{lemma}
\begin{proof}
Let $x\in X^{G}$. By (2) in Lemma \ref{lemma:linearization} there is an embedding
$G\hookrightarrow\GL(T_xX)$. The lemma follows from Theorem \ref{thm:Jordan-classic} applied
to the image of this embedding.
\end{proof}

\begin{lemma}
\label{lemma:oriented-actions} Let a finite group $G$ act
smoothly and preserving the orientation on a connected oriented
manifold $X$. For any $\gamma\in G$, any connected component of
the fixed point set $X^{\gamma}$ is a neat submanifold of even
codimension in $X$.
\end{lemma}
\begin{proof}
Combine Lemma \ref{lemma:linearization}
and the fact that for any $A\in\SO(n,\RR)$ the difference
$n-\dim\Ker(A-\Id)$ is even (note that
$X^{\gamma}$ is not necessarily connected, so it may have components of different
dimensions).
\end{proof}

\subsection{Finite group actions and cohomology}

In the following two lemmas we denote by $b_j(Y;k)$ the $j$-th
Betti number of a space $Y$ with coefficients in a field $k$.

\begin{lemma}
\label{lemma:Lefschetz}
Let $\Gamma$ be a finite cyclic group acting on a compact manifold $X$
and let $\gamma\in\Gamma$ be a generator. We have
\begin{equation}
\label{eq:Lefschetz-2}
\chi(X^{\Gamma})= \sum_j(-1)^j\Tr(H^j(\gamma):H^j(X;\QQ)\to H^j(X;\QQ)).
\end{equation}
In particular, if the action of $\Gamma$ on $H^*(X;\QQ)$ is trivial, then
$\chi(X^{\Gamma})=\chi(X).$
In general, 
\begin{equation}
\label{eq:cota-chi-fix}
|\chi(X^{\Gamma})|\leq \sum_j b_j(X;\QQ).
\end{equation}
\end{lemma}
\begin{proof}
Formula (\ref{eq:Lefschetz-2}) is classic, see Exercise 3 in \cite[Chap III, 6.17]{D}.
To prove (\ref{eq:cota-chi-fix}) note that,
since $\gamma$ has
finite order, all the eigenvalues of $H^j(\gamma):H^j(X;\QQ)\to H^j(X;\QQ)$ have
modulus one, so $|\Tr(H^j(\gamma):H^j(X;\QQ)\to H^j(X;\QQ))|\leq b_j(X;\QQ)$.
\end{proof}

\begin{lemma}
\label{lemma:betti-numbers-fixed-point-set}
Let $\Gamma\simeq\ZZ_p$ act on a manifold $X$.
Then
$$\sum_j b_j(X^{\Gamma};\FF_p)\leq \sum_j b_j(X;\FF_p).$$
\end{lemma}
\begin{proof}
This is \cite[Theorem III.4.3]{Bo}.
\end{proof}

\subsection{CT and CTO actions}
We say that the action of a group $G$ on a manifold $X$ is
cohomologically trivial (CT for short) if the induced action of
$G$ on $H^*(X;\ZZ)$ is trivial. If $X$ is orientable, then we
say that the action is CTO if it is CT and orientation
preserving (this makes sense without having to specify an
orientation, because a CT action preserves connected
components). Of course, if $X$ is closed and orientable then CT
implies CTO, but for manifolds with boundary this is not the
case.

\begin{lemma}
\label{lemma:Minkowski}
For any compact manifold $X$ and any
finite group $G$ acting on $X$ there is a subgroup $G_0\subseteq G$
such that $[G:G_0]$ is $X$-bounded and the action of $G_0$ on $X$ is CTO.
\end{lemma}
\begin{proof}
Since $X$ is compact,
it has finitely many components, so any group acting on $X$
has a subgroup of $X$-bounded index which acts preserving connected components
and orientation preservingly.
Furthermore, the cohomology of $X$ is finitely generated as an abelian group.
Let $T\subseteq H^*(X;\ZZ)$ be the torsion. A classic
result of Minkowski states that, given any integer $k$,
the size of any finite subgroup of $\GL(k;\ZZ)$ is $k$-bounded
(see \cite{Mi,S}). So if $G$ is a finite group acting on $X$, there is
a subgroup $G'\subseteq G$ of $X$-bounded index whose action on
$H^*(X;\ZZ)/T$ is trivial. There is also a subgroup
$G''\subseteq G'$ of index at most $|\Aut(T)|$ which acts trivially on $T$.
Let $F:=H^*(X;\ZZ)/T$. In terms of a splitting $H^*(X;\ZZ)\simeq F\oplus T$,
the action of $G''$ on $H^*(X;\ZZ)$ is through lower triangular matrices
with ones in the diagonal,
so it factors through the group $\Hom(F,T)$, which is finite; hence, there is a subgroup
$G_0\subseteq G''$ of index at most $|\Hom(F,T)|$ whose action on $H^*(X;\ZZ)$
is trivial.
\end{proof}

\section{Pseudofree actions: proof of Theorem \ref{thm:main-isolated}}
\label{s:pseudofree}

\subsection{The singular set and its projection to the orbit space}
Consider an arbitrary action of
a finite group $G$ on a compact manifold $X$. Recall that
the singular set of the action of $G$ on $X$ is
\begin{equation}
\label{eq:singular-set}
S_X=\bigcup_{g\in G\setminus\{1\}}X^g=\{x\in X\mid G_x\neq \{1\}\},
\end{equation}
Let $\pi:X\to Y:=X/G$ denote the projection to the orbit space, and
let $S_Y:=\pi(S_X)$.

\begin{lemma}
\label{lemma:Euler-X-wf-Y-wf}
The cohomologies of the spaces $Y$, $S_X$ and $S_Y$ are finitely generated
abelian groups, so $\chi(Y)$, $\chi(S_X)$
and $\chi(S_Y)$ are well defined. Furthermore,
we have
$$\chi(X)-\chi(S_X)=|G|(\chi(Y)-\chi(S_Y)).$$
\end{lemma}
\begin{proof}
Let $(\cC,\phi)$ be a $G$-regular triangulation of $X$. This
means that $\cC$ is a a $G$-regular finite simplicial complex
(in the sense of Definition 1.2 of \cite[Chapter III]{B} ---
note that the $G$-regularity of $\cC$ implies that $\cC/G$ is a
simplicial complex) and $\phi:X\to |\cC|$ is a $G$-equivariant
homeomorphism. Regular triangulations always exist: the second
barycentric subdivision of an arbitrary equivariant
triangulation of $X$ (which exists e.g. by \cite{I}) is
automatically regular (see Proposition 1.1 in \cite[Chapter
III]{B}).

The quotient $\cC/G$ is a simplicial complex and
the homeomorphism $\phi:X\to|\cC|$ descends to a homeomorphism
$\phi_Y:Y\to|\cC/G|$ (here we use the homeomorphism $|\cC|/G\simeq|\cC/G|$
described at the end of Section 1 in \cite[Chapter III]{B}). Hence
$H^*(Y;\ZZ)$ is a finitely generated abelian group, so $\chi(Y)$ is well defined.
Let $\cC'=\{\sigma\in\cC\mid G_{\sigma}\neq\{1\}\}$. The
regularity of $\cC$ implies that $\phi(S_X)=|\cC'|$ and
$\phi_Y(S_Y)=|\cC'/G|$, which imply that $\chi(S_X)$ and
$\chi(S_Y)$ are well defined.
Since Euler
characteristics can be computed counting simplices in triangulations,
we have
$$\chi(X)-\chi(S_X)=\sum_{\sigma\in\cC\setminus\cC'}(-1)^{\dim\sigma},
\qquad
\qquad
\chi(Y)-\chi(S_Y)=\sum_{[\sigma]\in(\cC/G)\setminus(\cC'/G)}(-1)^{\dim\sigma}.$$
Since $G$ acts freely on $\cC\setminus\cC'$ (and, of course, preserving dimensions),
we have
$$\sum_{\sigma\in\cC\setminus\cC'}(-1)^{\dim\sigma}=
|G|\left(\sum_{[\sigma]\in(\cC/G)\setminus(\cC'/G)}(-1)^{\dim\sigma}\right),$$
which proves the lemma.
\end{proof}

\subsection{Proof of Theorem \ref{thm:main-isolated}}
Consider a pseudofree effective action of a finite group $G$ on a compact connected manifold $X$
with nonzero Euler characteristic.
By Lemma \ref{lemma:Minkowski} we may replace $G$ by a subgroup of $X$-bounded
index and assume that $G$ acts trivially on $H^*(X;\ZZ)$.
By Lemma \ref{lemma:Lefschetz}, for any $\gamma\in G\setminus\{1\}$ the set
$X^{\gamma}$ consists of $\chi(X)$ points. This implies, if $\chi(X)<0$, that
$G=\{1\}$, so Theorem \ref{thm:main-isolated} is true in this case.

Let us assume for the rest of the proof that $\chi:=\chi(X)$ is positive.
Denote for convenience $d=|G|$. Since $S_X=\bigcup_{\gamma\in G\setminus\{1\}}X^{\gamma}$,
$$|S_X|\leq (d-1)\chi.$$
Let $Y=X/G$. By \cite[Chap II, Prop 9.13]{D} we have $H_*(X;\QQ)^G\simeq H_*(Y;\QQ)$,
so $\chi(Y)=\chi$.
Lemma \ref{lemma:Euler-X-wf-Y-wf} gives
$$|S_Y|=\frac{(d-1)\chi+|S_X|}{d}\leq \frac{2(d-1)\chi}{d}\leq 2\chi.$$
This implies that the number $r$ of $G$-orbits in $S_X$ is at most $2\chi$.
Let $d/a_1,\dots,d/a_r$ be the number of elements of the $G$-orbits in $S_X$,
and assume that $a_1\geq\dots\geq a_r$. Then $|S_X|=\sum d/a_j$, so
Lemma \ref{lemma:Euler-X-wf-Y-wf} implies that
$$\frac{d}{\chi a_1}+\dots+\frac{d}{\chi a_r}-1=\frac{d(r-\chi)}{\chi}.$$
The following lemma implies that $d/(\chi a_1)$ is $(\chi,r)$-bounded, hence $X$-bounded.

\begin{lemma}
\label{lemma:equacio-diofantina}
Suppose that $d,e_1,\dots,e_l,a$ are positive integers
satisfying: $e_1\geq \dots\geq e_l$, each $e_j$ divides $d$, and
\begin{equation}
\label{eq:equacio-diofantina}
\frac{d}{e_1}+\dots+\frac{d}{e_l}-1=\frac{dt}{a}.
\end{equation}
for some integer $t$.
Then $d/e_1$ is $(a,l)$-bounded.
\end{lemma}
\begin{proof}
Consider for any $(l,a)\in\NN^2$ the set
$\sS(l,a)\subset \NN^{l+1}\times\ZZ$ consisting of
tuples $(d,e_1,\dots,e_l,t)$ satisfying
(\ref{eq:equacio-diofantina}),
$e_1\geq\dots\geq e_l$ and $e_j|d$ for each $j$.
Define $D\colon \NN\times\NN\to\NN$ recursively
as follows: $D(1,a):=a$ and
$D(l,a):=\max\{D(l-1,aj)\mid 1\leq j\leq al\}$
for each $l>1$ (in fact $D(l,a)=D(l-1,a^2l)$).

We prove that for any $(d,e_1,\dots,e_l,t)\in\sS(l,a)$ we have
$e_1\geq d/D(l,a)$ using induction on $l$. For the
case $l=1$, suppose that $(d,e_1,t)\in\sS(1,a)$ and let $d=e_1g$,
where $g\in\NN$. Rearranging (\ref{eq:equacio-diofantina}) we deduce
that $g$ divides $a$, which implies $g\leq a$,
so $e_1=d/g\geq d/a=d/D(1,a)$.
Now assume that $l>1$ and that the inequality has been proved for smaller
values of $l$. Let $(d,e_1,\dots,e_l,t)\in\sS(l,a)$.
Since each $e_j$ divides $d$, we have $d/e_j\geq 1$ for each $j$, so the left
hand side in (\ref{eq:equacio-diofantina}) is positive. This implies that $t\geq 1$.
Using $e_1\geq\dots\geq e_l$ we can estimate $d/a\leq l d/e_l$,
so $1\leq e_l\leq al$. Furthermore, (\ref{eq:equacio-diofantina}) implies
$$\frac{d}{e_1}+\dots+\frac{d}{e_{l-1}}-1=\frac{dt}{a}-\frac{d}{e_l}=
\frac{d(te_l-a)}{ae_l},$$
so $(d,e_1,\dots,e_{l-1},te_l-a)$ belongs to $\sS(l-1,aj)$ for some $1\leq j\leq al$.
Using the induction hypothesis we deduce that $e_1\geq d/D(l-1,aj)\geq d/D(l,a)$.
\end{proof}

Let $x\in S_X$ be one of the points whose $G$-orbit has $d/a_1$ elements. Then
$[G:G_x]=d/a_1$ is $X$-bounded.
By Lemma \ref{lemma:fixed-point-Jordan} there is an abelian
subgroup $G_a\subseteq G_x$ of $X$-bounded index. Since $G_a$
is abelian and can be identified with a subgroup of $\GL(\dim
X,\RR)$ (see the proof of Lemma \ref{lemma:fixed-point-Jordan})
it follows that there exists a subgroup $G_b\subseteq G_a$ of
$(\dim X)$-bounded index which can be generated by $[\dim X/2]$
elements.
Let $\gamma\in G_b$ be any nontrivial element. Since
$X^{G_b}\subseteq X^{\gamma}$ and $X^{\gamma}$ consists of
$\chi$ points, the subgroup $A\subseteq G_b$ fixing each
element of $X^{\gamma}$ satisfies $[G_b:A]\leq\chi!$ and
$\gamma\in A$. The latter implies $X^A\subseteq X^{\gamma}$ so
$X^A=X^{\gamma}$. Since $A$ is a subgroup of an abelian
subgroup which can be generated by $[\dim X/2]$ elements, $A$
can also be generated by $[\dim X/2]$ elements. Finally, the
index $[G:A]$ is $X$-bounded, so the proof of Theorem
\ref{thm:main-isolated} is complete.

\section{Proof of Theorem \ref{thm:main}}
\label{s:main}

\subsection{Basic idea of the proof: $C$-rigid actions}
\label{ss:some-ideas}

Let $G$ be a finite group acting effectively on a compact,
connected and oriented $4$-manifold $X$ satisfying $\chi(X)\neq
0$. Roughly speaking, the proof of Theorem \ref{thm:main} is
based, as the proof of Theorem \ref{thm:main-isolated}, on
estimating the Euler characteristic of the singular set
$S_X=\bigcup_{g\in G\setminus\{1\}}X^g$ and deducing the
existence of some point $x\in X$ whose stabilizer has
$X$-bounded index $[G:G_x]$.

However, estimating in a useful way
$\chi(S_X)$ for general actions is much more difficult than in the case of pseudofree actions.
If the action is trivial in cohomology (which we may assume, replacing $G$ by a subgroup of
$X$-bounded index) then $\chi(X^g)=\chi(X)$ for every $g\in G$,
but to compute $\chi(S_X)$
(say, using the inclusion-exclusion principle) one needs to control the numbers
$\chi(X^{g_1}\cap \dots\cap X^{g_k})$ for different $g_1,\dots,g_k\in G$,
and there is no general formula for this quantity\footnote{However, for some restricted
classes of groups acting on $X$ one can study in detail the topology of the singular
set; in the case of minimal non-abelian groups,
this is done in \cite{Mc1,Mc2}, and it is the crucial ingredient of the proofs.}.

To circumvent this difficulty we replace the singular set $S_X$
by a set $S_X'\subset X$ whose Euler characteristic is much
easier to compute and which is in some sense a {\it uniform
approximation} of $S_X$; by the latter we mean that there exist
$X$-bounded constants, $1<C_1\leq C_2$, independent of $G$,
such that the isotropy group of any point in $S_X'$ (resp. in
the complementary of $S_X'$) has at least $C_1$ (resp. at most
$C_2$) elements. The actual definition of $S_X'$ uses the
notion of $C$-rigid subgroup of $G$, which we next explain (see
Section \ref{s:rigid} for more details).

Let $C$ be a natural number. We say that the action of a
subgroup $A\subseteq G$ is $C$-rigid if $A$ is abelian and for
any subgroup $A'\subseteq A$ satisfying $[A:A']\leq C$ we have
$X^{A'}=X^A$. Sometimes, abusing terminology, we simply
say that $A$ is $C$-rigid. The following two trivial properties
of $C$-rigidity will be implicitly used in our arguments.
First, if $C\leq C'$ and $A\subseteq G$ is a $C'$-rigid
subgroup then $A$ is also $C$-rigid. Second, if $A\subseteq G$ is
$C$-rigid, and $A_0\subseteq A$ is a subgroup, then $A_0$ is
$C_0$-rigid for any $C_0$ such that $C_0[A:A_0]\leq C$.

In Section \ref{s:rigid} we prove the following properties for
any finite group action $G$ on $X$:
\begin{itemize}
\item[{\it (a)}] there exists some $X$-bounded constant
    $C_{\chi}$ such that for any $C_{\chi}$-rigid subgroup
    $A\subseteq G$ we have $\chi(X^A)=\chi(X)\neq 0$ and
    each connected component of $X^A$ is even dimensional
    (Lemma \ref{lemma:Euler-characteristic});
\item[{\it (b)}] for any $C$ there exists a $(C,X)$-bounded
    constant $\Lambda_C$ such that any abelian subgroup
    $A\subseteq G$ has a $C$-rigid subgroup $A_0\subseteq
    A$ satisfying $[A:A_0]\leq \Lambda_C$ (Lemma
\ref{lemma:existence-rigid}); more precisely, $\Lambda_C$
will denote the minimal number with that property, and this
implies that $\Lambda_C$ is a nondecreasing function of
$C$.
\end{itemize}

To define $S_X'$ we take a suitable $X$-bounded number $C$ and
we set $S_X'=\bigcup_A X^A$, where $A$ runs over the set of
nontrivial $C$-rigid subgroups of $G$. This is an
approximation of $S_X$ in the previous sense: property
{\it (b)} and Jordan's theorem guarantees that if $x\in
X\setminus S_X'$ then $G_x$ can not be too big, whereas the
definition of rigidity implies that if $A$ is nontrivial and
$C$-rigid then $|A|>C$, from which we deduce that if $x\in
S_X'$ then $|G_x|>C$.

The actual definition of $C$ is given in formula
(\ref{eq:def-C}) below. The reader should think of $C$ as a big
but $X$-bounded number. The choice of $C$ guarantees that each connected component of $S_X'$ has
the same Euler characteristic as $X$. The action of $G$ on $X$
induces an action on the set of connected components of
$S_X'$, and we will prove that the number of $G$-orbits of
connected components of $S_X'$ is $X$-bounded (Lemma
\ref{lemma:G-orbits-in-hH}). From this we will deduce, using
the same arithmetic arguments as in the proof of Theorem
\ref{thm:main-isolated}, that there is some point in $X$
satisfying $[G:G_x]\leq C'$, where $C'$ is $X$-bounded.

\subsection{Details of the proof}
\label{ss:details-thm:main} The next three paragraphs are
devoted to proving some useful properties of rigid group
actions. The proof of Theorem \ref{thm:main} is in Subsection
\ref{sss:details-thm:main}.

\subsubsection{$\Jor_4$-rigid groups}
Here $\Jor_4$ refers to the constant in Jordan's Theorem
\ref{thm:Jordan-classic} for $n=4$.

\begin{lemma}
\label{lemma:restriccio-superficie} Suppose that $G$ is a
finite group acting on $X$ and that $A_1,A_2\subseteq G$ are
abelian subgroups satisfying $X^{A_1}\cap
X^{A_2}\neq\emptyset$. If $A_1$ is $\Jor_4$-rigid then there is
a subgroup $A_2'\subseteq A_2$ such that $[A_2:A_2']\leq\Jor_4$ and $A_2'$ preserves $X^{A_1}$.
\end{lemma}
\begin{proof}
Let $\Gamma=\la A_1,A_2\ra\subseteq G$. Since
$X^{\Gamma}=X^{A_1}\cap X^{A_2}\neq\emptyset$, Lemma
\ref{lemma:fixed-point-Jordan} implies that there exists an
abelian subgroup $H\subseteq\Gamma$ satisfying $[\Gamma:H]\leq
\Jor_4$. Since $[A_1:A_1\cap H]\leq\Jor_4$ and $A_1$ is
$\Jor_4$-rigid, we have $X^{A_1}=X^{A_1\cap H}$. Let
$A_2'=A_2\cap H$. Then $[A_2:A_2']\leq\Jor_4$. Finally, since
$H$ is abelian the action of $A_2'\subseteq H$ on $X$ preserves
$X^{A_1\cap H}=X^{A_1}$.
\end{proof}

\subsubsection{The constant $C$}
\label{sss:def-C} Define, for any compact manifold $Y$, the
following numbers
$$b_+(Y):=\sum_{j\geq 0}\max\{ b_j(Y;\FF_p)\mid p\text{ prime}\},
\qquad
b_-(Y):=\sum_{j\geq 0}\min\{ b_j(Y;\FF_p)\mid p\text{ prime}\}$$
and denote by $\sS(X)$ the set of diffeomorphism classes of
compact connected surfaces $\Sigma$ satisfying $b_-(\Sigma)\leq
b_+(X)$. Abusing language, we will sometimes say that a surface
belongs to $\sS(X)$ meaning that its diffeomorphism type
belongs to $\sS(X)$. Note that $\sS(X)$ is never empty, as it
always contains $S^2$.

By Lemma \ref{lemma:surface-Ghys-chi} (which is the analogue
for surfaces of Theorem \ref{thm:main}), for any compact
surface $\Sigma$ there exists a $\Sigma$-bounded natural number
$C(\Sigma)$ such that any finite group $G$ acting effectively
on $\Sigma$ has an abelian subgroup $A$ satisfying $[G:A]\leq
C(\Sigma)$ and $\chi(\Sigma^A)=\chi(\Sigma)$. The
classification theorem of compact connected surfaces implies
that
$$C_{\surf}=\max\{C(\Sigma)\mid \Sigma\text{ compact surface},\,\Sigma\in\sS(X)\}$$
is finite and $X$-bounded. This is well defined because
$\sS(X)\neq\emptyset$, and we have $C_{\surf}\geq 1$.

\newcommand{\tra}{\operatorname{tra}}

By Lemma \ref{lemma:X-fixed-points} there is some $X$-bounded
constant $C_f$ such that for any finite abelian subgroup $A$
acting on $X$ the fixed point set $X^A$ has at most $C_f$
connected components.
Recall that $C_{\chi}$ denotes an $X$-bounded constant with the property
that if a finite group $G$ acts effectively on $X$ and $A\subseteq G$
is any $C_{\chi}$-rigid subgroup then $\chi(X^{A})=\chi(X)$ and each
connected component of $X^A$ is even dimensional (see Lemma \ref{lemma:Euler-characteristic}).
Let
$$C_{\delta}=\max\{C_{\chi},C_fC_{\surf}\}.$$
The following lemma shows part of the significance of the number $C_{\delta}$.

\begin{lemma}
\label{lemma:centralitzador-interseca} Let $G$ be a finite
group acting on $X$ in a CTO way, and let $A_1,A_2\subseteq G$
be $C_{\delta}$-rigid subgroups. If the intersection $A_1\cap
A_2$ is nontrivial then $X^{A_1}\cap X^{A_2}\neq\emptyset$.
\end{lemma}
\begin{proof}
Let $a\in A_1\cap A_2$ be a nontrivial element. Since the
action of $G$ is CTO, by Lemma \ref{lemma:Lefschetz} we have
$\chi(X^a)=\chi(X)$ and by Lemma \ref{lemma:oriented-actions}
each connected component of $X^a$ is even dimensional. Choose a
connected component $Y\subseteq X^a$ satisfying $\chi(Y)\neq
0$. The group $A_i$ preserves $X^a$, and the subgroup
$A_i'\subseteq A_i$ preserving $Y$ satisfies $[A_i:A_i']\leq
C_f$. Since $A_i$ is $C_{\delta}$-rigid we have
$X^{A_i'}=X^{A_i}$. If $Y$ is a point then $\emptyset\neq
Y\subseteq X^{A_1'}\cap X^{A_2'}$, hence $X^{A_1}\cap
X^{A_2}\neq\emptyset$. If $Y$ is a surface then $Y\in\sS(X)$
(Lemma \ref{lemma:surface-punts-fixes}) so by Lemma
\ref{lemma:surface-Ghys-chi} the group $\Gamma=\la
A_1',A_2'\ra$ (which acts on $Y$) has an abelian subgroup
$A\subseteq\Gamma$ of index $[\Gamma:A]\leq C_{\surf}$ and
satisfying $\chi(Y^A)=\chi(Y)\neq 0$, so $Y^A\neq\emptyset$. We
have $X^{A_i'\cap A}=X^{A_i}$ because of $C_{\delta}$-rigidity
so
$$X^{A_1}\cap X^{A_2}\subseteq X^{\Gamma}\subseteq X^A\subseteq
X^{A_1'\cap A}\cap X^{A_2'\cap A}$$ implies $X^{A_1}\cap X^{A_2}=X^A$, and $X^A\neq\emptyset$ because
$Y^A\subseteq X^A$.
\end{proof}

Let $\Lambda_{\delta}$ be the $X$-bounded number, given by Lemma
\ref{lemma:existence-rigid}, with the property that any abelian
finite group $A$ acting on $X$ has a $C_{\delta}$-rigid
subgroup $A_0$ of index at most $\Lambda_{\delta}$. Define the
following constant:
\begin{equation}
\label{eq:def-C}
C:=\max\{C_{\chi},\Jor_4C_fC_{\surf},\Jor_4\Lambda_{\delta},2\Jor_4\}.
\end{equation}
The expression in the right hand side is redundant, since
$\Lambda_{\delta}$ can not be smaller than $C_{\chi}$; we
include the constant $C_{\chi}$ inside the maximum for clarity.

\subsubsection{Properties of $C$-rigid groups}
In the following two lemmas we prove that $C$-rigid subgroups
of finite groups acting on $X$ have particularly nice
properties.

\begin{lemma}
\label{lemma:transitivitat} Suppose that $G$ is a finite group
acting on $X$ and that $A_1,A_2,A_3\subseteq G$ are $C$-rigid
subgroups satisfying $X^{A_1}\cap X^{A_2}\neq\emptyset\neq
X^{A_1}\cap X^{A_3}$. Then $X^{A_2}\cap X^{A_3}\neq\emptyset$.
\end{lemma}
\begin{proof}
By Lemma \ref{lemma:restriccio-superficie} there exist
subgroups $A_2'\subseteq A_2$ and $A_3'\subseteq A_3$
satisfying $[A_j:A_j']\leq \Jor_4$ such that both $A_2'$ and
$A_3'$ preserve $X^{A_1}$. Let $\Gamma=\la
A_2',A_3'\ra\subseteq G$. Since $A_1$ is $C_{\chi}$-rigid,
there is a connected component $Y\subseteq X^{A_1}$ such that
$\chi(Y)\neq 0$. Since $X^{A_1}$ has at most $C_f$ connected
components, the subgroup $\Gamma'\subseteq\Gamma$ preserving
$Y$ satisfies $[\Gamma:\Gamma']\leq C_f$. The subvariety $Y$ is
even dimensional, so it is either a point or an element of
$\sS(X)$ (Lemma \ref{lemma:surfaces-in-X}). In the first case
we have $Y\subseteq X^{A_2'\cap\Gamma'}\cap
X^{A_3'\cap\Gamma'}=X^{A_2}\cap X^{A_3}$, the second equality
following from $[A_j:A_j'\cap\Gamma']\leq C$ ($j=2,3$) and
rigidity. If $Y\in\sS(X)$ then by Lemma
\ref{lemma:surface-Ghys-chi} there is a subgroup
$\Gamma''\subseteq\Gamma'$ satisfying $[\Gamma':\Gamma'']\leq
C_{\surf}$ and $\chi(Y^{\Gamma''})=\chi(Y)\neq 0$, so
$Y^{\Gamma''}\neq\emptyset$. So $Y^{\Gamma''}\subseteq
X^{A_2'\cap\Gamma''}\cap X^{A_3'\cap\Gamma''}=X^{A_2}\cap
X^{A_3}$, again because of $[A_j:A_j'\cap\Gamma'']\leq C$ and
rigidity.
\end{proof}

\begin{lemma}
\label{lemma:common-nice} Suppose that $G$ is a finite group
acting on $X$ and that $A_1,\dots,A_r\subseteq G$ are $C$-rigid
subgroups (with $r$ arbitrary) satisfying $X^{A_1}\cap
X^{A_j}\neq\emptyset$ for every $j$. Let
$Z:=X^{A_1}\cap\dots\cap X^{A_r}$. Then there is a
$C_{\delta}$-rigid subgroup $A\subseteq G$ such that $Z=X^A$.
In particular (since $C_\delta\geq C_\chi$), $\chi(Z)=\chi(X)$
(so $Z\neq\emptyset$), every connected component of $Z$ is even
dimensional, and $Z$ has at most $C_f$ connected components.
\end{lemma}
\begin{proof}
We first prove that $Z\neq\emptyset$. Assume that $r\geq 2$,
otherwise the claim is trivial. The proof of the claim is very
similar to that of the previous lemma. For every $j\geq 2$
there exists a subgroup $A_j'\subseteq A_j$ satisfying
$[A_j:A_j']\leq \Jor_4$ such that $A_j'$ preserves $X^{A_1}$.
Let $\Gamma=\la A_2',\dots,A_r'\ra\subseteq G$. Let $Y\subseteq
X^{A_1}$ be a connected component such that $\chi(Y)\neq 0$.
The subgroup $\Gamma'\subseteq\Gamma$ preserving $Y$ satisfies
$[\Gamma:\Gamma']\leq C_f$. If $Y$ is a point then setting
$A_j'':=A_j'\cap\Gamma'$ we have $Y\subseteq
X^{A_2''}\cap\dots\cap X^{A_r''}$ and $[A_j:A_j'']\leq C$, so
by rigidity $X^{A_j''}=X^{A_j}$ for every $j$, which implies
that $Z=X^{A_1''}\cap\dots\cap X^{A_r''}$, and we are done. If
$Y\in\sS(X)$ then there is a subgroup
$\Gamma''\subseteq\Gamma'$ satisfying $[\Gamma':\Gamma'']\leq
C_{\surf}$ and $\chi(Y^{\Gamma''})=\chi(Y)\neq 0$, so
$Y^{\Gamma''}\neq\emptyset$. Setting $A_j'':=A_j'\cap\Gamma''$
we have $Y^{\Gamma''}\subseteq X^{A_2''}\cap\dots\cap
X^{A_r''}$ and $[A_j:A_j'']\leq C$, and the proof is finished
as in the case where $Y$ is a point.

Let $T=\la A_1,\dots,A_r\ra\subseteq G$. We have
$X^{T}=Z\neq\emptyset$, so by Lemma
\ref{lemma:fixed-point-Jordan} there is an abelian subgroup
$H\subseteq T$ of index at most $\Jor_4$. By Lemma
\ref{lemma:existence-rigid} there is a $C_{\chi}$-rigid
subgroup $H_{\chi}\subseteq H$ satisfying $[H:H_{\chi}]\leq
\Lambda_{\chi}$. Let $A_j'=A_j\cap H_{\chi}$. Then
$[A_j:A_j']\leq C$, so $X^{A_j}=X^{A_j'}$ for every $j$ because
$A_j$ is $C$-rigid. Thus
$$Z=X^{\Gamma}\subseteq
X^{H_{\chi}}\subseteq X^{A_1'}\cap\dots\cap X^{A_r'}=Z,$$
which implies $Z=X^{H_{\chi}}$. Hence $A:=H_{\chi}$ has the
desired property.
\end{proof}

\subsubsection{The sets $\fF$ and $\hH$ and the proof of Theorem \ref{thm:main}}
\label{sss:details-thm:main}
Recall our assumptions: $G$ is a
finite group acting effectively on a compact, oriented and
connected $4$-manifold $X$ satisfying $\chi(X)\neq 0$. By Lemma
\ref{lemma:Minkowski} we may (and do) replace $G$ by a subgroup
of $X$-bounded index whose action on $X$ is CTO. Let $C$ be the
constant defined in (\ref{eq:def-C}) above. Define the
following collection of (not necessarily connected)
submanifolds of $X$:
$$\fF=\{X^A\mid A\subseteq G\text{, $A$ is nontrivial and $C$-rigid}\}.$$
The action of $G$ on $X$ induces an action on $\fF$, since
$g\,X^A=X^{gAg^{-1}}$ for any $g\in G$ and $A\subseteq G$ is $C$-rigid if and only
if $gAg^{-1}$ is. Let
$\approx$ be the relation between elements of $\fF$ which
identifies $F,F'\in\fF$ whenever $F\cap F'\neq\emptyset$. By
Lemma \ref{lemma:transitivitat} this is an equivalence
relation. Let
$$\hH:=\fF/\approx.$$
The action of $G$
on $\fF$ preserves the relation $\approx$, so it descends to an action on $\hH$.

We are going to prove in Lemma \ref{lemma:G-orbits-in-hH} below
that $|\hH/G|$ is $X$-bounded. Before that, we introduce some
notation and a preliminary result (Lemma
\ref{lemma:molts-isoladors}). For any $H\in\hH$ define
$$X_H:=\bigcap_{F\in\fF,\,[F]=H}F.$$
Then $\{X_H\mid H\in\hH\}$ is a collection of disjoint (not
necessarily connected) submanifolds of $X$. By Lemma
\ref{lemma:common-nice} we have $\chi(X_H)=\chi(X)\neq 0$, all
connected components of $X^H$ are even dimensional, and $X_H$
has at most $C_f$ connected components (recall that $C_f$ is
defined a few lines before Lemma
\ref{lemma:centralitzador-interseca}).

We denote by $G_H\subseteq G$ the isotropy group of each
element $H\in\hH$. The action of $G_H$ on $X$ preserves $X_H$.
Let
$$G(H)=\{g\in G\mid \exists Y\subset X\text{ such that $Y$ is a connected component of $X_H$
and of $X^g$}\}.$$
Note that unlike $G_H$ the subset $G(H)$ is not a subgroup of $G$.

\begin{lemma}
\label{lemma:molts-isoladors} There exist an $X$-bounded number
$C_1$ such that for any $H\in\hH$ satisfying $|G_H|>C_1$ we
have $|G(H)|\geq |G_H|/(2C_1)$.
\end{lemma}
\begin{proof}
Choose a connected component $Y\subseteq X_H$ satisfying
$\chi(Y)\neq\emptyset$. Since $X_H$ has at most $C_f$ connected
components there is a subgroup $G_H'\subseteq G_H$ whose action
on $X_H$ preserves $Y$ and such that $[G_H:G_H']\leq C_f$.
Recall that $Y$ is even dimensional.

Suppose that $Y$ consists of a unique point $y\in X$. Then
$G_H'\subseteq G_y$, so $|G_y|\geq |G_H'|\geq |G_H|/C_f$. By
Jordan's Theorem \ref{thm:Jordan-classic} there is an abelian
subgroup $B\subseteq G_y$ satisfying $[G_y:B]\leq\Jor_4$. Since
$B$ acts linearly on $T_yX$ there is a splitting
$T_yX=L_1\oplus L_2$ where $L_1,L_2$ are $2$-dimensional linear
subspaces of $T_yX$ preserved by the action of $B$. Let
$B_j\subseteq B$ be the subgroup fixing all vectors of $L_j$.
We prove that $|B_1|\leq |B|/3$. If this were not true, then
the subgroup $G_1\subseteq G_y$ fixing all vectors of $L_1$
would satisfy $[G_y:G_1]\geq 2\Jor_4$. For any $C$-rigid
subgroup $A\subseteq G_y\subseteq G$ we would have
$X^A=X^{A\cap G_1}$ since $[A:A\cap G_1]\leq [G_y:G_1]\leq
2\Jor_4\leq C$, which would imply by Lemma
\ref{lemma:linearization} that $T_yX^A$ contains $L_1$. But
since $y$ is an isolated point of $X_H$, there must exist some
nontrivial $C$-rigid subgroup $A\subseteq G$ such that
$L_1\nsubseteq T_yX^A$, a contradiction. Hence $|B_1|\leq
|B|/3$, and for the same reason $|B_2|\leq |B|/3$. Now
$B\setminus (B_1\cup B_2)\subseteq G(H)$ so $|G(H)|\geq
|B\setminus (B_1\cup B_2)|\geq |B|/3\geq |G_y|/(3\Jor_4)\geq
|G_H|/(3\Jor_4 C_f)$.

Now suppose that $\dim Y=2$ so that $Y\in \sS(X)$.
By Lemma \ref{lemma:surface-Ghys-chi} there is a subgroup
$G_H''\subseteq G_H'$ such that
$\chi(Y^{G_H''})=\chi(Y)\neq 0$ and such that $[G_H':G_H'']\leq
C_{\surf}$. We have $X^{G_H''}\neq\emptyset$, so
by Lemma \ref{lemma:fixed-point-Jordan} there is an abelian
subgroup $A\subseteq G_H''$ such that $[G_H'':A]\leq \Jor_4$.
By Lemma \ref{lemma:existence-rigid} there exists an
$X$-bounded number $\Lambda$ with the property that $A$ has a
$C$-rigid subgroup $A_0\subseteq A$ satisfying
$[A:A_0]\leq\Lambda$. Let
$$C_1':=C_fC_{\surf}\Jor_4\Lambda.$$
Then $[G_H:A_0]\leq C_1'$. If $|G_H|>C_1'$ then $A_0$ is
nontrivial, so $X^{A_0}$ is an element of $\fF$. Since
$\emptyset\neq Y^{G_H''}\subseteq X^{A_0}$, the $\approx$-class
of $X^{A_0}$ is $H$, hence $X_H\subseteq X^{A_0}$. Since $Y$ is
$2$-dimensional and all connected components of $X^{A_0}$ are
even dimensional, it follows that $A_0\setminus\{1\}\subseteq
G(H)$. Since $A_0\neq\{1\}$ we have $|G(H)|\geq
|A_0\setminus\{1\}|\geq |A_0|/2\geq |G_H|/(2C_1')$.

Hence setting $C_1:=\max\{C_1',3\Jor_4 C_f\}$ the lemma holds
true.
\end{proof}

\begin{lemma}
\label{lemma:G-orbits-in-hH} $|\hH/G|\leq C_1+2C_1C_f$.
\end{lemma}
\begin{proof}
Choose for each $H\in\hH$ a $C_{\delta}$-rigid subgroup
$A(H)\subseteq G$ such that $X_H=X^{A(H)}$. If $H\neq H'$ are
elements of $\hH$, then since $X^{A(H)}\cap
X^{A(H')}=\emptyset$ we have $A(H)\cap A(H')=\{1\}$ by Lemma
\ref{lemma:centralitzador-interseca}. Consequently,
$|\hH|=s\leq |G|$. Let
$$\hH_{\small}=\{H\in\hH\mid |G_H|\leq C_1\},\qquad
\hH_{\big}=\{H\in\hH\mid |G_H|>C_1\}.$$
Both subsets $\hH_{\small},\hH_{\big}\subseteq\hH$ are $G$-invariant.
Each $G$-orbit in $\hH_{\small}$ has at least $|G|/C_1$ elements, so the bound
$|\hH_{\small}|\leq|\hH|\leq |G|$ implies that $\hH_{\small}$ contains at most
$C_1$ orbits, i.e., $|\hH_{\small}/G|\leq C_1$.
To estimate the number of orbits in $\hH_{\big}$ we use the following:
$$|G|\cdot |\hH_{\big}/G|=\sum_{H\in\hH_{\big}}|G_H|
\leq 2C_1\sum_{H\in \hH_{\big}}|G(H)|\leq 2C_1C_f|G|.$$ The
equality follows from a simple counting argument, the first
inequality follows from Lemma \ref{lemma:molts-isoladors}, and
the second inequality follows from the fact that the
submanifolds $\{X_H\}$ are disjoint and that for any $g\in G$
the number of connected components of $X^g$ is at most $C_f$.
Dividing both extremes by $|G|$ we deduce $|\hH_{\big}/G|\leq
2C_1C_f$ which combined with the estimate on $|\hH_{\small}/G|$
proves the lemma.
\end{proof}

For any $H\in \hH$ let
$$Y_H=\bigcup_{F\in\fF,\,[F]=H}H.$$
By Lemma \ref{lemma:common-nice} and the inclusion-exclusion
principle, we have $\chi(Y_H)=\chi(X).$ Let $H_1,\dots,H_r\in
\hH$ be representatives of the orbits of the action of $G$ on
$\hH$. Let $$d=|G|$$ and let $e_j=|G_{H_j}|$. Since for
different $H,H'\in\hH$ we have $Y_H\cap Y_{H'}=\emptyset$, we
have
\begin{equation}
\label{eq:chi-de-fF}
\chi\left(\bigcup_{F\in\fF}F\right)=
\left(\frac{d}{e_1}\chi(Y_{H_1})+\dots+\frac{d}{e_r}\chi(Y_{H_r})\right)=
\chi(X)\left(\frac{d}{e_1}+\dots+\frac{d}{e_r}\right).
\end{equation}

\begin{lemma}
\label{lemma:divisible} The difference
$\chi(X)-\chi\left(\bigcup_{F\in\fF}F\right)$ is divisible by
$d/a$, where $a$ is an $X$-bounded divisor of $d$.
\end{lemma}
\begin{proof}
Let $\Lambda$ be the same $X$-bounded number as in the proof of
Lemma \ref{lemma:molts-isoladors}. Let $x\in X$. If $|G_x|>
\Jor_4\Lambda$ then, by Jordan's Theorem
\ref{thm:Jordan-classic} and Lemma \ref{lemma:existence-rigid},
there is a nontrivial $C$-rigid subgroup $A\subseteq G_x$, so $x\in
X^A\subseteq \bigcup_{F\in\fF}F$. So the isotropy group of any
point in $X\setminus\bigcup_{F\in\fF}F$ has at most
$\Jor_4\Lambda$ elements. Now take a $G$-regular triangulation
of $X$ (see the proof of Lemma \ref{lemma:Euler-X-wf-Y-wf}).
The regularity of the triangulation implies that the
isotropy group of any simplex is contained in the isotropy
group of any of its points.
Hence each orbit of
simplexes in $X\setminus\bigcup_{F\in\fF}F$ has size $d/e$, where
$e$ is a divisor of $d$ and $e\leq\Jor_4\Lambda$, so $e$ divides
$a:=\GCD(d,(\Jor_4\Lambda)!)$.
Now
$\chi(X)-\chi\left(\bigcup_{F\in\fF}F\right)$ can be computed
as the alternate sum of numbers of simplexes in each dimension
which are not contained in
$\bigcup_{F\in\fF}F$. Grouping the simplexes in $G$-orbits, the result
follows immediately.
\end{proof}

Combining the previous lemma with (\ref{eq:chi-de-fF}) we obtain the following
equality:
$$\frac{d}{e_1}+\dots+\frac{d}{e_r}-1=\frac{dt}{a},$$
where $t$ is an integer and $a$ is an $X$-bounded divisor of
$d$. By Lemma \ref{lemma:G-orbits-in-hH}, $r$ is also
$X$-bounded. We can assuming (reordering if necessary) that
$e_1\geq\dots\geq e_r$. Lemma \ref{lemma:equacio-diofantina}
gives $|G_{H_1}|=e_1\geq d/K$ for some constant $K$ depending
only on $r$ and $a$, so $K$ is $X$-bounded. It follows that
$[G:G_{H_1}]=d/e_1\leq K$ is $X$-bounded.

By the arguments in the proof of Lemma
\ref{lemma:molts-isoladors}, there is a subgroup $G_2\subseteq
G_{H_1}$ of $X$-bounded index such that
$X_{H_1}^{G_2}\neq\emptyset$.
By Lemma \ref{lemma:fixed-point-Jordan} and Lemmas
\ref{lemma:existence-rigid} and
\ref{lemma:Euler-characteristic} there exists an abelian
subgroup $A\subseteq G_2$ of $X$-bounded index such that
$\chi(X^{A})=\chi(X)$, and $A$ can be generated by $2$
elements.
If $\chi(X)<0$ then, since $\chi(X^A)<0$, there is at least one connected
component of $X^A$ which is a surface. If $\Sigma\subseteq X^A$ is one such component
and $x\in\Sigma$, then the linearization of the action of $A$ near $x$ gives an embedding
$A\hookrightarrow \GL(T_xX/T_x\Sigma)$ preserving the orientation and a metric (see
Lemma \ref{lemma:linearization}),
so we may identify $A$ with a subgroup of $\SO(2,\RR)$; hence $A$ is cyclic.

Since $[G:A]$ is $X$-bounded, the proof of Theorem
\ref{thm:main} is now complete.

\section{Finite groups acting on surfaces}
\label{s:surfaces}

In this section we consider finite group actions on surfaces.
The main result is Lemma \ref{lemma:surface-Ghys-chi}, which is
the analogue in two dimensions of Theorem \ref{thm:main}.

\begin{lemma}
\label{lemma:surface-punts-fixes}
Let $\Sigma$ be a compact connected surface. For any
finite abelian group $A$ acting on $\Sigma$ the number of connected
components of $\Sigma^A$ is $\Sigma$-bounded.
\end{lemma}
\begin{proof}
It clearly suffices to consider nontrivial actions. So
let $A$ be a finite abelian group acting on $\Sigma$ and assume
that there is an element $a\in A$ acting nontrivially on $\Sigma$.
We distinguish two possibilities.

If all connected components of $\Sigma^a$ are zero dimensional
then, by (\ref{eq:cota-chi-fix}) in Lemma \ref{lemma:Lefschetz}
we have $|\Sigma^a|=\chi(\Sigma^a)\leq
b(\Sigma;\QQ):=\sum_{j=0}^2b_j(\Sigma;\QQ)$. Since
$\Sigma^A\subseteq\Sigma^a$, the result follows.

Now assume that $\Sigma^a$ contains some one-dimensional
component. Any such component is (diffeomorphic to)
either a circle or a closed interval.
For $j=0,1$ let $\Sigma^a_j\subseteq \Sigma^a$ denote the union
of the connected components whose Euler characteristic is $j$.
By (\ref{eq:cota-chi-fix}) in Lemma \ref{lemma:Lefschetz} we have
$|\pi_0(\Sigma^a_1)|\leq b(\Sigma;\QQ)$. Let us now
bound $|\pi_0(\Sigma^a_0)|$, which is equal to the number of circles
in $\Sigma^a$. The fact that $\Sigma^a$ has a codimension
one connected component implies, by (1) in Lemma \ref{lemma:linearization},
that $a$ has order $2$. Let $\la a\ra=\{1,a\}$. Then
$\Sigma':=\Sigma/\la a\ra$ is a surface with corners, so it is homeomorphic
to a surface with boundary. We may bound
$$\chi(\Sigma')=(\chi(\Sigma)+|\pi_0(\Sigma^a_1)|)/2\geq \chi(\Sigma)/2$$
using an $A$-regular triangulation on $\Sigma$ (see the proof of Lemma \ref{lemma:Euler-X-wf-Y-wf})
and computing Euler characteristics in terms of counting simplices.
As a topological surface, $\Sigma'$ is the complementary in a compact
connected surface $S$ of finitely many disjoint
open discs; $\chi(\Sigma')$ is equal to $\chi(S)$ minus the number of
discs, and the latter can be identified with $|\pi_0(\partial \Sigma')|$.
By the classification of compact connected surfaces
we have $\chi(\Sigma)\leq 2$; this gives $\chi(\Sigma')\leq 2-|\pi_0(\partial\Sigma')|$ or, equivalently,
$$|\pi_0(\partial \Sigma')|\leq 2-\chi(\Sigma')$$
Each connected component
of $\Sigma^a_0$ contributes
to a connected component of $\partial\Sigma'$.
We deduce that $|\pi_0(\Sigma_0^a)|\leq 2-\chi(\Sigma)/2$.

To complete the argument in this case, note that $\Sigma^A\subseteq\Sigma^a$.
This implies that $\Sigma^A$ contains at most as many one-dimensional connected
components as $\Sigma^a$, so we only need to bound the number of zero dimensional connected
components (i.e., the isolated points) of $\Sigma^A$. Each isolated point in $\Sigma^A$
is either an isolated point in $\Sigma^a$ or belongs to a one-dimensional connected
component of $\Sigma^a$. Since we have a bound on $|\pi_0(\Sigma^a)|$, it suffices
to bound uniformly the number of isolated points in $\Sigma^A$ which can belong to
a given one-dimensional component of $\Sigma^a$. If $Y\subseteq\Sigma^a$ is one
such component and $Y$ contains an isolated point of $\Sigma^A$, then the action of
$A$ on $\Sigma^a$ preserves $Y$ and we can identify $\Sigma^A\cap Y$ with
the fixed point set of the action of $A$ on $Y$.
To finish the proof it suffices to check that $Y^A$ contains at most $2$ points.
Let $g\in A$ be an element acting nontrivially on $Y$.
Then $Y^g$ is a finite set of points, and
$|Y^g|\leq b_0(Y;\QQ)+b_1(Y;\QQ)\leq 2$ by (\ref{eq:cota-chi-fix}) in Lemma \ref{lemma:Lefschetz}. Since
$Y^A\subseteq Y^g$, the result follows.
\end{proof}

\begin{lemma}
\label{lemma:surface-chi}
For any compact connected surface $\Sigma$
and any finite abelian group $A$ acting on
$\Sigma$ there is an abelian subgroup $A_0\subseteq A$ such that
$[A:A_0]$ is $\Sigma$-bounded and $\chi(\Sigma^{A_0})=\chi(\Sigma)$.
\end{lemma}
\begin{proof}
Let $\Sigma$ be a compact connected surface, and let an abelian group $A$ act on $\Sigma$.
By Lemma \ref{lemma:Minkowski} there exists a subgroup $A'\subseteq A$
whose action on $\Sigma$ is CT and $[A:A']$ is $\Sigma$-bounded.
If the action of $A'$ on $\Sigma$ is trivial, then
we set $A_0:=A'$ and we are done. Otherwise, there exists some $a\in A'$ acting
nontrivially on $\Sigma$. By Lemma \ref{lemma:Lefschetz}, $\chi(\Sigma^a)=\chi(\Sigma)$.
By Lemma \ref{lemma:surface-punts-fixes} the number of connected components of $\Sigma^a$
is $\Sigma$-bounded. It follows that there exists a subgroup $A_0\subseteq A$ of $\Sigma$-bounded
index whose action on $\Sigma^a$ preserves each connected component and is orientation preserving on
each component of $\Sigma^a$.
We claim that $\chi(\Sigma^{A_0})=\chi(\Sigma^a)$.
To prove this, it suffices to check that for any connected component $Y\subseteq\Sigma^a$
we have $\chi(Y)=\chi(Y^{A_0})$. But each such $Y$ is a closed manifold of dimension at most $1$,
so $\chi(Y)=\chi(Y^{A_0})$ follows from Lemma \ref{lemma:Lefschetz} and the fact that $A_0$ acts on $Y$ preserving the orientation.
\end{proof}

\begin{lemma}
\label{lemma:surface-Ghys-chi}
For any compact connected surface $\Sigma$ and any finite group $G$ acting
effectively on $\Sigma$ then there is an abelian subgroup $A\subseteq G$
whose index $[G:A]$ is $\Sigma$-bounded and which satisfies $\chi(\Sigma^A)=\chi(\Sigma)$.
\end{lemma}
\begin{proof}
In view of Lemma \ref{lemma:surface-chi} it suffices to check that,
for any compact connected surface $\Sigma$, any finite group acting effectively on
$\Sigma$ has an abelian subgroup of $\Sigma$-bounded index.
To prove this, suppose first that $\partial\Sigma$ is empty. If $\Sigma$ is orientable,
then the lemma is Theorem 1.3 in \cite{M1} (if furthermore $\chi(\Sigma)\neq 0$ then
it also follows from Theorem \ref{thm:main-isolated} and Lemma
\ref{lemma:oriented-actions} of the present paper). If $\Sigma$ is not orientable,
then the arguments of Section 2.3 in \cite{M1} allow to deduce the lemma
from the orientable case.
Now suppose that $\partial\Sigma$ is nonemtpy, say with $k$ connected
components. Let a finite group $G$ act on $\Sigma$. Replacing $G$ by
a subgroup of index at most $k$, we can assume that $G$ fixes one
connected component $Y\subset\partial\Sigma$. Considering the restriction of the action
to $Y$ we get a morphism of groups $G\to\Diff(Y)$ which we claim to be
injective. This follows from the fact that a finite order diffeomorphism
of $\Sigma$ which is the identity on $Y$ is automatically the identity
on the whole $\Sigma$, which in turn is a consequence of (1.b) in Lemma
\ref{lemma:linearization}. So to finish the proof we need to prove that a finite
subgroup of $\Diff(S^1)$ has an abelian subgroup of uniformly bounded
index. This the simplest case of Theorem 1.4 in \cite{M1}, but it can
also be proved directly observing that,
since all metrics in $S^1$ are isometric up to rescaling,
choosing an invariant metric on $S^1$ gives an embedding
of the group in a dihedral group.
\end{proof}

\section{$C$-rigid group actions on $4$-manifolds}
\label{s:rigid}

In this section we prove some facts on finite group actions on
compact $4$-manifolds and on rigidity that were used in Section
\ref{s:main} when proving Theorem \ref{thm:main}.

\subsection{Bounding the number of components of fixed point sets}

The following notation, which is recalled for convenience, was
defined in Subsection \ref{sss:def-C}. For any space $Y$ with
finitely generated homology we set $b_+(Y):=\sum_{j\geq
0}\max\{ b_j(Y;\FF_p)\mid p\text{ prime}\}$ and
$b_-(Y):=\sum_{j\geq 0}\min\{ b_j(Y;\FF_p)\mid p\text{
prime}\}$. For any $4$-dimensional oriented manifold $X$ we
denote by $\sS(X)$ the set of diffeomorphism classes of compact
connected surfaces $\Sigma$ such that $b_-(\Sigma)\leq b_+(X)$.

\begin{lemma}
\label{lemma:surfaces-in-X} Let $X$ be a $4$-dimensional
compact connected oriented manifold $X$, and let $H$ be a group
acting nontrivially on $X$ preserving the orientation. The
connected components of $X^H$ are neat submanifolds of
dimensions $0$, $1$ or $2$. Any two-dimensional connected
component of $X^H$ is diffeomorphic to an element of $\sS(X)$.
\end{lemma}
\begin{proof}
That $X^H$ is a (not necessarily connected) neat submanifold of
$X$ follows from (1.b) in Lemma \ref{lemma:linearization}. By
Lemma \ref{lemma:oriented-actions}, for any $h\in H$ the
connected components of $X^h$ are zero or two-dimensional;
hence, the dimension of any connected component of $X^H$ is at
most two. To prove the last statement, suppose that $Y\subset
X^H$ is a two-dimensional connected component. Let $h\in H$ be
an element acting nontrivially; replacing $h$ by a power $h^r$
we may assume that the diffeomorphism of $X$ induced by the
action of $h$ has primer order.  Since the connected components
of $X^h$ have dimension at most $2$, the inclusion $X^H\subset
X^h$ implies that $Y$ is a connected component of $X^h$. Then,
by Lemma \ref{lemma:betti-numbers-fixed-point-set}, $b_-(Y)\leq
b_+(X)$, so $Y$ is diffeomorphic to an element of $\sS(X)$.
\end{proof}

\begin{lemma}
\label{lemma:X-fixed-points} For any compact $4$-dimensional
oriented manifold $X$ and any finite abelian group $A$ acting
on $X$ the number of connected components of $X^A$ is
$X$-bounded.
\end{lemma}
\begin{proof}
Let $X$ be a $4$-dimensional oriented manifold. Let $A$ be a
finite abelian group acting on $X$. If the action of $A$ is
trivial then there is nothing to prove. Otherwise, let $a\in A$
be an element acting nontrivially on $X$ through a
diffeomorphism of order $p$, where $p$ is any prime. By Lemma
\ref{lemma:betti-numbers-fixed-point-set} we have
$$\sum_j b_j(X^a;\FF_p)\leq \sum_j b_j(X;\FF_p)\leq b_+(X),$$
so $X^a$ has at most $b_+(X)$ connected components, and each
connected component $Y\subseteq X^a$ satisfies $b_-(Y)\leq
b_+(X)$. Since $X^A\subseteq X^a$, it suffices to prove that
for connected component of $X^a$ contains an $X$-bounded
amount of connected components of $X^A$. By Lemma
\ref{lemma:oriented-actions} the connected components of $X^a$
are either points or surfaces. Of course each isolated point in
$X^a$ contains at most one connected component of $X^A$. Now
suppose that $Y\subseteq X^a$ is a surface. Then $Y$ is
diffeomorphic to some element of $\sS(X)$. If $Y\cap
X^a=\emptyset$, then there is nothing to prove. Otherwise, the
action of $A$ on $X^a$ leaves $Y$ fixed. By Lemma
\ref{lemma:surface-punts-fixes}, the number of connected
components of $Y^{A}$ i $Y$-bounded. Since $Y$ is diffeomorphic
to an element of $\sS(X)$, the argument is finished using the
classification of compact surfaces, which implies that for
every $N$ the set of diffeomorphism types of compact surfaces $\Sigma$
satisfying $b_-(\Sigma)\leq N$ is finite.
\end{proof}

\begin{lemma}
\label{lemma:inclusion-chains} For any compact $4$-dimensional
oriented manifold $X$, and any chain of inclusions
$\emptyset\neq Y_1\subsetneq Y_2\subsetneq\dots\subsetneq Y_r$
of neat\footnote{See \cite[\S 1.4]{H} for the definition of neat submanifold.}  submanifolds of $X$  satisfying $|\pi_0(Y_j)|\leq k$ for each $j$, we have
$r\leq \left(5+k\atop 5\right).$
\end{lemma}
\begin{proof}
This is a particular case of Lemma 7.1 in \cite{M2}.
\end{proof}

\subsection{Definition and basic results on  $C$-rigid abelian group actions}
Let $A$ be a finite group acting on a compact $4$-manifold $X$
and let $C$ be a natural number. Recall (see Subsection
\ref{ss:some-ideas}) that (the action of) $A$ is said to be
$C$-rigid if $A$ is abelian and for any subgroup $A_0\subseteq
A$ satisfying $[A:A_0]\leq C$ we have $X^{A_0}=X^A$.

\begin{lemma}
\label{lemma:existence-rigid} Let $X$ be a compact connected
$4$-manifold. For any natural number $C$ there exists a
$(C,X)$-bounded constant $\Lambda$ such that any finite abelian
group $A$ acting on $X$ has a subgroup of index at most
$\Lambda$ whose action on $X$ is $C$-rigid.
\end{lemma}
\begin{proof}
By Lemma \ref{lemma:X-fixed-points} there is an $X$-bounded
constant $C_f$ such that for any finite abelian group $A$
acting on $X$ the fixed point set $X^A$ has at most $C_f$
connected components. Let $C':=\left(5+C_f\atop 5\right).$ We
prove that $\Lambda:=C^{C'-1}$ has the stated property. Let $A$
be a finite abelian group acting on $X$ in a CTO way and assume
by contradiction that no subgroup of $A$ of index at most
$\Lambda$ is $C$-rigid. Then we may construct recursively a
sequence of subgroups $A=:A_0\supset A_1\supset\dots\supset
A_{C'}$ satisfying $[A_i:A_{i+1}]\leq C$ and $X^{A_i}\subset
X^{A_{i+1}}$ for each $i$; indeed, once $A_0,A_1,\dots,A_i$,
$i<C'$, have been constructed we have $[A:A_i]\leq C^i\leq
C^{C'-1}$ so by our initial assumption on $A$ the group $A_i$
is not $C$-rigid; hence, we may pick a subgroup $A_{i+1}\subset
A_i$ such that $[A_i:A_{i+1}]\leq C$ and $X^{A_i}\subset
X^{A_{i+1}}$. By Lemma \ref{lemma:X-fixed-points}, each
$X^{A_i}$ has at most $C_f$ connected components, so we obtain
a contradiction with Lemma \ref{lemma:inclusion-chains}.
\end{proof}

\begin{lemma}
\label{lemma:Euler-characteristic} Let $X$ be a compact
connected $4$-manifold. There exists an $X$-bounded constant
$C_{\chi}$ such that any finite abelian group $A$ acting on $X$
in a $C_{\chi}$-rigid way satisfies $\chi(X^A)=\chi(X)$ and
each connected component of $X^A$ is even dimensional.
\end{lemma}
\begin{proof}
It suffices to prove that any finite abelian group $A$ acting
on $X$ has a subgroup $A'$ of $X$-bounded index such that
$\chi(X)=\chi(X^{A'})$ and each connected component of $X^{A'}$
is even dimensional. So suppose that $A$ is a finite abelian
group acting on $X$. By Lemma \ref{lemma:Minkowski} we may take
a subgroup $A_1\subseteq A$ of $X$-bounded index whose action
on $X$ is CTO. If $A_1$ acts trivially on $X$ then we set
$A':=A_1$ and we are done. Otherwise there exists some $a\in
A_1$ whose action on $X$ is nontrivial. By Lemma
\ref{lemma:Lefschetz} $\chi(X^a)=\chi(X)$ and by Lemma
\ref{lemma:X-fixed-points} the number of connected components
of $X^a$ is $X$-bounded. Hence the subgroup $A_2\subseteq A_1$
preserving each connected component of $X^a$ and whose action
on each connected component of $X^a$ is orientation preserving
has $X$-bounded index $[A_1:A_2]$. By Lemma
\ref{lemma:surface-chi}, Lemma \ref{lemma:surfaces-in-X}, and
the classification of compact surfaces, there is a subgroup
$A_3\subseteq A_2$ of $X$-bounded index such that for every
two-dimensional connected component $Y$ of $X^a$ we have
$\chi(Y^{A_3})=\chi(Y)$. We may clearly assume that $a\in A_3$.
Since the action of $A_3$ on each two-dimensional connected
component $Y\subseteq X^a$ is orientation preserving, $Y^{A_3}$
is even dimensional. For every zero-dimensional connected
component $Y\subseteq X^a$ we obviously have
$\chi(Y^{A_3})=\chi(Y)$. Since $a$ acts on $X$ preserving the
orientation, each connected component of $X^a$ has dimension
$0$ or $2$, by Lemma \ref{lemma:oriented-actions}. It then
follows, as in the proof of Lemma \ref{lemma:surface-chi}, that
$\chi(X^{A_3})=\chi(X^a)=\chi(X)$ and that each connected
component of $X^A$ is even dimensional.
\end{proof}

\end{document}